\documentclass[11pt]{amsart}
\usepackage[foot]{amsaddr}
\usepackage{amsmath,amssymb,amsthm,framed}
\usepackage{color,graphics,srcltx,float,anysize}
\usepackage{stmaryrd}
\usepackage{enumerate}
\usepackage{booktabs}
\usepackage[table,x11names]{xcolor}
\usepackage{longtable}
\usepackage{makecell}
\usepackage{hyperref}
\usepackage{soul}
\usepackage{colortbl}
\usepackage{caption}
\usepackage{bbold}

\marginsize{1.9cm}{1.9cm}{1.9cm}{1.9cm}

\usepackage{xcolor}
\hypersetup{
    colorlinks,
    linkcolor={red!50!black},
    citecolor={blue!50!black},
    urlcolor={blue!80!black}
}

\makeatletter
\newcommand\newtag[2]{#1\def\@currentlabel{#1}\label{#2}}
\makeatother

\definecolor{darkblue}{rgb}{0,0,0.8}

\newtheorem{theorem}{Theorem}[]
\newtheorem{lemma}[theorem]{Lemma}
\newtheorem{corollary}[theorem]{Corollary}

\newtheorem{remark}[theorem]{Remark}

\theoremstyle{definition}

\newtheorem{problem}[theorem]{Problem}
\numberwithin{theorem}{section}

\newcommand{\qi}{\mathbb{Q}\mathrm{I}}

\DeclareMathOperator{\PSL}{{\mathrm{PSL}}}

\DeclareMathOperator{\GF}{{\mathrm{GF}}}

\DeclareMathOperator{\Aut}{{\mathrm{Aut}}}

\DeclareMathOperator{\vLS}{{\mathrm{vLS}}}

\newcommand{\vo}{\mathrm{VO}}
\newcommand{\vd}{\mathrm{VD}}
\newcommand{\vsz}{\mathrm{VSz}}
\newcommand{\fis}{\mathrm{Fi}}
\newcommand{\no}{\mathrm{NO}}
\newcommand{\nug}{\mathrm{NU}}

\renewcommand{\leq}{\leqslant}
\renewcommand{\geq}{\geqslant}
\renewcommand{\le}{\leqslant}
\renewcommand{\ge}{\geqslant}

\usepackage{color}

\allowdisplaybreaks

\title{Separating rank 3 graphs
}
\author[Bamberg]{John Bamberg}
\author[Giudici]{Michael Giudici}
\author[Lansdown]{Jesse Lansdown}
\author[Royle]{Gordon Royle}
\thanks{The third author is now at the University of Canterbury, Christchurch 8140, New Zealand.}

\address{Centre for the Mathematics of Symmetry and Computation, Department of Mathematics and Statistics, The University of Western Australia, Perth, WA 6009, Australia. The third author is now at the University of Canterbury, Christchurch 8140, New Zealand}
\email{firstname.lastname@uwa.edu.au}

\begin{document}

\begin{abstract}
We classify, up to some notoriously hard cases, the rank 3 graphs which fail to meet either the Delsarte or the Hoffman bound. As a consequence, we resolve the question of separation for the corresponding rank 3 primitive groups and give new examples of synchronising, but not $\qi$, groups of affine type.
\end{abstract}

\maketitle

\section{Introduction}

Finite primitive permutation groups of rank 3 have attracted a lot of interest since the work of Donald Higman in the early 1960s, resulting in the discovery of new sporadic simple groups and new combinatorial methods for permutation groups (e.g., coherent configurations). Over a long sequence of papers \cite{Bannai:1971,KantorLiebler:1982,Liebeck:1987,LiebeckSaxl}, the classification of rank 3 primitive groups was completed as a by-product of the Classification of Finite Simple Groups. 
A transitive group has \emph{rank 3} when it has precisely three orbits on ordered pairs. If there are also three orbits on unordered pairs, a graph can be formed with edges and non-edges defined by the two non-diagonal orbits, respectively. A 
(non-complete) graph is called a \emph{rank 3 graph} if it can be produced in this way for some group of rank 3. In fact, the automorphism group of a rank 3 graph must necessarily have rank 3.

The rank 3 condition is a strong symmetry condition, for it means that the
automorphism group acts transitively on edges and non-edges. This implies that such graphs are \emph{strongly regular}: precisely the connected graphs having three distinct eigenvalues.
A surprising amount of information can be determined from these eigenvalues; for instance, Delsarte \cite{Delsarte:1973aa} showed that the clique number of a strongly regular graph is bounded in terms of its minimum eigenvalue. {If $\Gamma$ is a strongly-regular graph with maximum clique size $\omega(\Gamma)$ and maximum coclique size $\alpha(\Gamma)$ then the clique-coclique bound states that 
\[
\alpha(\Gamma) \omega(\Gamma) \leqslant |V(\Gamma)|.
\]}
{We  call a graph \emph{separating} if the product of its clique and coclique numbers is \emph{strictly less} than its order -- in other words, if equality in the clique-coclique bound does not hold.}  This terminology is motivated by the following connection to \emph{separating groups}:
Neumann's separation theorem (\cite{Neumann1976}, see also \cite{BBMN76}) for permutation groups states that if a group $G$
acts transitively on a finite set $\Omega$, and if $A$ and $B$ are two subsets of $\Omega$ such that $|A|\cdot |B| < |\Omega|$, then there exists $g\in G$
that \emph{separates} the two subsets: that is $A^g\cap B=\varnothing$.
Motivated by this, a group is called \emph{separating} if for any two non-singleton subsets $A, B \subseteq \Omega$ where $|A||B|= |\Omega|$, there exists $g \in G$ such that $A^g \cap B = \varnothing$.
We can describe separation of a transitive group $G$ via its invariant graphs \cite[Theorem 5.4]{Araujo:2017aa}: $G$ is 
separating if and only if every nontrivial $G$-invariant graph is separating.
It is known that a separating group must be primitive \cite[Theorem 3.6 \& Corollary 5.5]{Araujo:2017aa}. Determining which primitive groups are separating has been identified as an important problem by Araujo et al., and has led to many interesting combinatorial applications (cf. \cite{ABC, Araujo:2017aa}).

The main result of this paper is Theorem \ref{thm:SepartingTransitiveSRGs}, which provides an almost complete classification of separating rank 3 graphs.

\begin{theorem} \label{thm:SepartingTransitiveSRGs}
Let $\Gamma$ be a rank 3 graph.  
Then $\Gamma$ is separating if it or its complement is one of the following:
\begin{enumerate}[\rm (1)]
            \item The triangular graph, $T(n)$, for $n \geqslant 5$, $n$ odd;
            \item The collinearity graph of a polar space in Table \ref{table:NoOvoids};
            \item A connected component of the distance-2 graph of the dual polar graph arising from a polar space of rank $5$ and order $(q,1)$;
            \item $\nug_m(2)$, for $m>3$;
            \item $\no^\epsilon_{2m}(q)$, for $\epsilon = \pm 1$, $m\geq 3$ and $q \in \{2, 3\}$;
            \item $\no^+_{2m+1}(q)$, for $m\geq 2$, and $q \in \{4, 8\}$;
            \item $\no^-_{2m+1}(q)$, for $m\geq 2$,  $q \in \{3, 4, 8\}$ and $(m, q) \neq(2, 3)$;
            \item The Grassmann graph $J_q(n,2)$, for $n\geq 5$, $n$ odd;
            \item $E_{6,1}(q)$;
    \item $\no^+_{2m+1}(3)$, for $m\geq 3$;
    \item The van Lint-Schrijver graph, $\vLS(p,e,t)$, for $t$ even;
    \item $\vo_{2m}^-(q)$, for $m\geq 2$;
    \item $\vd_{5,5}(q)$;
    \item $\vsz(2^{2e+1})$, for $e\ge 0$;
    \item $\Gamma$ belongs to Table \ref{table:NonFamilyGraphsSeparating};
\end{enumerate}
or is possibly one of the following unresolved cases:
\begin{enumerate}[\rm (I)]
    \item The collinearity graph of a polar space not listed in Table \ref{table:Ovoids} or Table \ref{table:NoOvoids};
    \item $\vo_{2m}^+(q)$, for $m>3$;
    \item The Peisert graph $P^*(p^{2t})$, for $t$ even.
\end{enumerate}
\end{theorem}

Each of the unresolved cases are notorious open problems in finite geometry and graph theory.
The first is equivalent to the existence of \textbf{ovoids} of finite polar spaces \cite[Theorem 6.8]{Araujo:2017aa}, where there are still many long-standing open problems. Orel \cite[Theorem 2, Proposition 5]{Orel} shows that $\vo_{2m}^+(q)$ is non-separating if and only if the polar space $Q^+_{2m+1}(q)$ has an ovoid, and so the second unresolved
case is also about ovoids of finite polar spaces. Finally,
the clique number of the self-complementary Peisert graph $P^*(p^{2r})$ is known to be equal to the square root 
of its order when $r$ is odd \cite[Theorem 5.1]{KP2004}, but for $r$ even, it is mostly an
open problem.

{\footnotesize
\begin{table}[!ht]
\centering
\begin{tabular}{ rll } 
 \toprule
& Graph & Parameters  \\
 \midrule
2 & $G_2(2)$-graph & $(36,14,4,6)$ \\ 
3 & Hoffman-Singleton & $(50, 7, 0, 1)$  \\ 
4 & Gewirtz & $(56, 10, 0, 2)$ \\
5 & $M_{22}$-graph & $(77, 16, 0, 4)$ \\ 
6 & Higman-Sims & $(100, 22, 0, 6)$ \\ 
7 & Hall-Janko & $(100, 36, 14, 12)$ \\ 
8 & $S_{10}$-graph  & $(126, 25, 8, 4)$ \\ 
9 & $U_4(3)$-graph & $(162, 56, 10, 24)$  \\ 
11 & $M_{22}$-graph & $(176, 70, 18, 34)$ \\ 
12 &Berlekamp-van Lint-Seidel & $(243, 22, 1, 2)$ \\ 
13 & Delsarte dual of Berlekamp-van Lint-Seidel & $(243,110,37,60)$  \\ 
14 & $M_{23}$-graph & $(253, 112, 36, 60)$ \\ 
15 & $2^8.S_{10}$-graph & $(256, 45, 16, 6)$ \\ 
17 & $2^8.L_2(17)$-graph & $(256,102,38,42)$ \\ 
18 & McLaughlin & $(275, 112, 30, 56)$ \\ 
21 &  $G_2(4)$-graph & $(416, 100, 36, 20)$ \\ 
29 & Dodecad-graph & $(1288, 792, 476, 504)$ \\ 
30 & Conway & $(1408, 567, 246, 216)$ \\ 
32 & Suzuki & $(1782, 416, 100, 96)$  \\ 
33 & $2^{11}.M_{24}$-graph & $(2048, 276, 44, 36)$  \\ 
34 & $2^{11}.M_{24}$-graph & $(2048, 759, 310, 264)$  \\
36 & Conway  & $(2300, 891, 378, 324)$  \\ 
37 & $7^4:(6.O_5(3))$-graph & $(2401, 240, 59, 20)$ \\ 
39 & $7^4:(3 \times 2.S_7)$-graph & $(2401,720,229,210)$ \\
41 & $\fis_{22}$-graph & $(3510, 693, 180, 126)$ \\
42 & Rudvalis & $(4060, 1755, 730, 780)$ \\ 
43 & $2^{12}.\mathit{HJ}.S_3$-graph & $(4096, 1575, 614, 600)$  \\ 
48 & $\fis_{22}$-graph & $(14080, 3159, 918, 648)$ \\ 
49 & $5^6.4.\mathit{HJ}.2$-graph & $(15625, 7560, 3655, 3660)$   \\ 
50 & $\fis_{23}$-graph & $(31671, 3510, 693, 351)$  \\
51 & $\fis_{23}$-graph & $(137632,28431, 6030, 5832)$ \\ 
52 & $\fis_{24}$-graph & $(306936, 31671, 3510, 3240)$ \\ 
53 & Suz-graph & $(531441,65520,8559,8010)$ \\
 \bottomrule
\end{tabular}
\caption{Rank $3$ separating graphs not belonging to an infinite family (numbered according to Table \ref{table:NonFamilyGraphs}).}
\label{table:NonFamilyGraphsSeparating}
\end{table}
}

\begin{corollary}\label{cor:sep}
If $\Gamma$ is a graph listed as one of (1)--(15) in Theorem \ref{thm:SepartingTransitiveSRGs} and $G\leqslant \Aut(\Gamma)$ has rank 3 then $G$ is separating, synchronising, and primitive. If $\Gamma$ is a rank 3 graph not listed as one of (1)--(15) in Theorem \ref{thm:SepartingTransitiveSRGs} and $G\leqslant \Aut(\Gamma)$, then $G$ is non-separating.
\end{corollary}

A description of the automorphism groups of the rank 3 graphs can be found in \cite{Skresanov}.
As a consequence of Corollary \ref{cor:sep} and the classification of rank $3$ primitive groups, we essentially know the separating groups of rank $3$, modulo notoriously difficult cases.
Separation belongs to a hierarchy of group properties involving the action of permutation groups on sets, multisets, and partitions. The hierarchy includes $\mathbb{Q}I$, spreading, separating, synchronising, and primitivity, where each of these properties is implied by the preceding property, but the implications do not reverse in general. We refer to \cite{Araujo:2017aa} for an introduction of the synchronisation hierarchy of permutation groups.
For affine primitive groups, $\mathbb{Q}I$ and spreading are equivalent \cite[Theorem 7.7]{Araujo:2017aa} and synchronisation and separation are equivalent \cite{HallPaigeConjecture}. We point out that Huang et al.~\cite{HuangAlt} show that the alternating forms graph $\mathrm{Alt}(5,p^m)$ is a core and so its rank 3 automorphism group is an affine primitive group that is  synchronising (\cite[Theorem 2.4]{Cameron:2008aa}), but not $\mathbb{Q}I$ (by \cite[Theorem 3]{Dixon:2005aa} since the sub orbits have different sizes and so the group is not $\frac{3}{2}$-transitive).
As a consequence of our results, together with the classification of affine $\qi$ groups \cite{Dixon:2005aa}, we give more examples of affine primitive permutation groups that are synchronising but not $\qi$.

For primitive groups of almost simple type, separating and synchronising are not equivalent. It is known that the group ${\rm P\Omega}(5,q)$ is synchronising but not separating for $q$ an odd prime \cite[Example 6.9]{Araujo:2017aa}. This is currently the only known infinite family. It follows from \cite[Corollary 4.5]{Araujo:2017aa} that a rank $3$ primitive group is non-synchronising when the chromatic number $\chi$ of the rank 3 graph (or its complement) is equal to its clique number. It is well known that $\chi(\Gamma) \geqslant \omega(\Gamma)$, and so it is natural to ask when synchronisation and separation are distinct properties of a rank 3 graph.

\begin{problem}
Given a non-separating rank 3 graph, $\Gamma$, when is $\chi(\Gamma) \neq \omega(\Gamma)$?
\end{problem}

\section{Background}

\subsection{Strongly regular graphs}

A graph $\Gamma$ is called \emph{strongly regular} with parameters $(\nu, k, \lambda, \mu)$ if (i) it has $\nu$ vertices, (ii) it has degree $k$, (iii) every pair of adjacent vertices has $\lambda$ common neighbours, and (iv) every pair of non-adjacent vertices has $\mu$ common neighbours. We refer to \cite{BrouwerHvM} for a detailed reference on strongly regular graphs. A (connected) strongly regular graph has precisely three eigenvalues: the degree $k$, and two others, which we shall denote by $s$ and $r$ throughout this paper, where $s<r$. The values of $s$ and $r$ can be computed directly from the parameters of the strongly regular graph by the following equations:

\[
s = \frac{\lambda - \mu - \sqrt{(\lambda - \mu)^2 + 4(k-\mu)}}{2}, \quad r = \frac{\lambda - \mu + \sqrt{(\lambda - \mu)^2 + 4(k-\mu)}}{2}.
\]

The size of cliques and cocliques in strongly regular graphs can both be bounded in terms of the eigenvalues. Denoting the maximum size of a clique and coclique of $\Gamma$ respectively by $\omega(\Gamma)$ and $\alpha(\Gamma)$ (or simply $\omega$ and $\alpha$ if the context is clear), we have the following two important bounds:

\begin{theorem}[Delsarte Bound \cite{Delsarte:1973aa}]\label{thm:DelsarteBound}
Given a strongly regular graph, $\Gamma$, with parameters $(\nu, k, \lambda, \mu)$,\[
\omega(\Gamma) \leqslant 1 -\frac{k}{s}.
\]
\end{theorem}

\begin{theorem}[Hoffman Bound cf. \cite{HoffmansRatioBound}] \label{thm:HoffmanBound}
Given a strongly regular graph, $\Gamma$, with parameters $(\nu, k, \lambda, \mu)$,
\[
\alpha(\Gamma) \leqslant \frac{\nu s}{s-k}.
\]
\end{theorem}

Note that the Hoffman bound is more general than stated in Theorem \ref{thm:HoffmanBound}, since it applies to all regular graphs. Moreover, the Delsarte bound can actually be derived from the Hoffman bound, although it was developed first. For a discussion on these bounds, and a history of their development, see \cite{HoffmansRatioBound}.

For a vertex-transitive or strongly regular graph, $\omega(\Gamma)\alpha(\Gamma)\leq \nu$ (cf. \cite{GodsilMeagher}).
Recall that a graph $\Gamma$ is \emph{separating} if $\omega(\Gamma)\alpha(\Gamma) < \nu$. For a strongly regular graph, and hence for a rank 3 graph, non-separating means that both the Delsarte and Hoffman bounds are sharp. The following lemma is important for this paper:
\begin{lemma}\label{Delsarte}
Let $\Gamma$ be a strongly regular graph with parameters $(\nu, k, \lambda, \mu)$. If $\omega(\Gamma) < 1 - \frac{k}{s}$ or $\alpha(\Gamma) < \frac{\nu s}{s - k}$
then $\Gamma$ is separating. In particular, if either $\frac{k}{s}$ or 
$\frac{\nu s}{s - k}$
is not an integer, then 
$\Gamma$ is separating. Conversely, $\Gamma$ is non-separating if and only if $\omega(\Gamma) = 1 - \frac{k}{s}$ and $\alpha(\Gamma) = \frac{\nu s}{s - k}$.
\end{lemma}

\begin{proof}
By Theorem \ref{thm:DelsarteBound}, we have $\omega(\Gamma) \leqslant 1 - \frac{k}{s}$ and by Theorem \ref{thm:HoffmanBound} we have
$\alpha(\Gamma) \leqslant 
\frac{\nu s}{s - k}$. The product of the right-hand sides is equal to $\nu$, and so 
$\Gamma$ is non-separating if and only if $\omega(\Gamma) = 1 - \frac{k}{s}$ and $\alpha(\Gamma) = \frac{\nu s}{s - k}$.
\end{proof}

\subsection{Rank 3 primitive groups}

Recall that rank 3 graphs are constructed as the orbital graphs of transitive rank 3 groups.
Moreover, a transitive permutation group $G$ of rank 3 is separating if and only if its corresponding rank 3 graph is separating, since this is the unique (up to complements) non-trivial $G$-invariant graph. Every separating group is primitive of \emph{affine}, \emph{almost simple}, or \emph{diagonal} type \cite[Theorem 3.2, Proposition 3.3 \& Corollary 5.4]{Araujo:2017aa}. However, the diagonal type case does not occur as a group of automorphisms of a rank 3 graph, since a diagonal type group
can never be rank 3, see for example \cite[p.7]{Cameron1981}. Hence a separating rank 3 graph must admit a primitive automorphism group of rank 3 of almost simple or affine type.

The primitive rank 3 groups have been completely classified (cf. \cite{LiebeckSaxl}). As a consequence, the rank 3 graphs with primitive automorphism groups are also known. We give this list in Theorem \ref{thm:AlmostSimpleAndAffineGraphs} for those admitting groups of automorphisms of affine or almost simple type.
This information has been collated from \cite{BrouwerHvM} where the classification is stated with respect to the groups. 
Note that rank 3 groups of grid type (indeed any group of wreath product type) are known not to be separating by \cite[Proposition 3.7]{Araujo:2017aa}.
References to the relevant sections of \cite{BrouwerHvM} are provided for constructions and properties of these graphs.

\begin{theorem}[cf. {\cite[Theorems 11.3.1, 11.3.2, 11.3.3, 11.3.4, 11.4.1]{BrouwerHvM}}]\label{thm:AlmostSimpleAndAffineGraphs}
Let $\Gamma$ be a strongly regular graph admitting a primitive rank $3$ group of automorphisms of almost simple or affine type. Then $\Gamma$ is either one of the special cases listed in Table \ref{table:NonFamilyGraphs}, or it belongs to one of the following families:
\begin{enumerate}[\rm (1)]
\item The triangular graph, $T(n)$, for $n\geqslant 4$; \hfill \cite[1.1.7]{BrouwerHvM}
\item The collinearity graph of a finite classical polar space or the dual of a finite classical polar space, with rank at least $2$ or rank exactly $2$, respectively; \hfill \cite[Thm 2.2.12 \& 2.2.19]{BrouwerHvM}
    \item A connected component of the distance-2 graph of the dual polar graph arising from a polar space of rank $5$ and order $(q,1)$; \hfill \cite[Thm 2.2.20]{BrouwerHvM} 
    \item $\nug_m(2)$, for $m> 3$; \hfill \cite[\S 3.1.6]{BrouwerHvM}
    \item $\no^\epsilon_{2m}(q)$, for $\epsilon = \pm 1$, $m\geq 3$, and $q \in \{2, 3\}$; \hfill \cite[\S 3.1.2]{BrouwerHvM} 
    \item $\no^{\epsilon}_{2m+1}(q)$, for $\epsilon = \pm 1$, $m\geq 2$ and $q \in \{3, 4, 8\}$; \hfill \cite[\S 3.1.4]{BrouwerHvM} 
    \item The Grassmann graph $J_q(n,2)$ for $n\geq 4$; \hfill \cite[\S 3.5.1]{BrouwerHvM}
    \item $E_{6,1}(q)$; \hfill \cite[\S 4.9]{BrouwerHvM}
        \item The Paley graph, $P_q$; \hfill \cite[\S 1.1.9]{BrouwerHvM}
        \item The Peisert graph, $P^*(p^{2t})$; \hfill \cite[\S 7.3.6]{BrouwerHvM}
        \item The van Lint-Schrijver graph, $\vLS(p,e,t)$; \hfill \cite[\S 7.3.1]{BrouwerHvM}
        \item The $n \times n$ grid; \hfill \cite[\S 1.1.8]{BrouwerHvM}
        \item The Bilinear forms graph $H_q(2, m)$; \hfill \cite[\S 3.4.1]{BrouwerHvM} 
        \item $\vo^\epsilon_{2m}(q)$; \hfill \cite[\S 3.3.1]{BrouwerHvM} 
        \item The alternating forms graph, $\mathrm{Alt}(5,p^m)$; \hfill \cite[\S 3.4.2]{BrouwerHvM} 
        \item The affine half spin graph, $\vd_{5,5}(q)$; \hfill \cite[\S 3.3.3]{BrouwerHvM}
        \item $\vsz(q)$, for 
         $q=2^{2e-1}$.  \hfill \cite[\S 3.3.1]{BrouwerHvM} 
\end{enumerate}
\end{theorem}

Some families overlap for certain parameters, and so we omit duplicates from our list. For example, $\no_3^+(q)$ is a rank 3 graph for all prime powers $q$. However, it is not included here since it is the triangular graph $T(q+1)$ which already belongs to the first family listed. Many graphs admit multiple primitive rank 3 groups of automorphisms, sometimes even of different type, so the list of rank 3 graphs is somewhat simpler than the list of groups.

{\footnotesize
\begin{table}[!ht]
\centering
\begin{tabular}{ rp{5cm}lllp{2.6cm} } 
 \toprule
& Graph & Parameters & \multicolumn{1}{c}{$s$, $r$} & $1 - \frac{k}{s}$, $\frac{\nu s}{s - k}$ & $\omega$, $\alpha$ \\
 \midrule
\newtag{1}{tab:S10.13} &  $S_{8}$-graph & $(35, 16, 6, 8)$ & $-4$, $2$& $5$, $7$ & 5, 7 \cite[\S10.13]{BrouwerHvM} \\ 
\rowcolor{black!10} \newtag{2}{tab:S10.14} & $G_2(2)$-graph & $(36,14,4,6)$ & $-4$, $2$ & ${9}/{2}$, $8$ & $3$, $7$ \cite[\S10.14]{BrouwerHvM} \\ 
\rowcolor{black!10} \newtag{3}{tab:S10.19} & Hoffman-Singleton & $(50, 7, 0, 1)$ & $-3$, $2$ & ${10}/{3}$, $15$ & $2$, $15$ \cite[\S10.19]{BrouwerHvM} \\ 
\rowcolor{black!10} \newtag{4}{tab:S10.20} & Gewirtz & $(56, 10, 0, 2)$ & $-4$, $2$ & ${7}/{2}$, $16$ & $2$, $16$ \cite[\S10.20]{BrouwerHvM}\\
\rowcolor{black!10} \newtag{5}{tab:S10.27} & $M_{22}$-graph & $(77, 16, 0, 4)$ & $-6$, $2$ & ${11}/{3}$, $21$ & $2$, $21$ \cite[\S10.27]{BrouwerHvM}\\ 
\rowcolor{black!10} \newtag{6}{tab:S10.31} & Higman-Sims & $(100, 22, 0, 6)$ & $-8$, $2$ & ${15}/{4}$, ${80}/{3}$ & $2$, $22$ \cite[\S10.31]{BrouwerHvM}\\ 
\rowcolor{black!20} \newtag{7}{tab:S10.32} & Hall-Janko & $(100, 36, 14, 12)$ & $-4$, $6$ & $10$, $10$ & $4$, $10$ \cite[\S10.32]{BrouwerHvM}\\
\rowcolor{black!10} \newtag{8}{tab:S10.40} & $S_{10}$-graph  & $(126, 25, 8, 4)$ & $-3$, $7$ & ${28}/{3}$, ${27}/{2}$ & 6, 12 \cite[\S10.40]{BrouwerHvM} \\
\rowcolor{black!10} \newtag{9}{tab:S10.48} & $U_4(3)$-graph & $(162, 56, 10, 24)$ & $-16$, $2$ & ${9}/{2}$, $36$ & $3$, $21$ \cite[\S10.48]{BrouwerHvM} \\ 
 \newtag{10}{tab:ExtraspecialClasses3} & Action of $S_4$ on $\rm{PG}(1, 13)$  & $(169, 72, 31, 30)$ & $-6$, $7$ & $13$, $13$ & $13$, $13$ OA Graph\\ 
\rowcolor{black!10} \newtag{11}{tab:S10.51} & $M_{22}$-graph & $(176, 70, 18, 34)$ & $-18$, $2$ & ${44}/{9}$, $36$ & $4$, $16$ \cite[\S10.51]{BrouwerHvM}\\ 
\rowcolor{black!10} \newtag{12}{tab:S10.55a} &Berlekamp-van Lint-Seidel & $(243, 22, 1, 2)$ & $-5$, $4$ & $27/5$, 45 & $3$, $45$ \cite[\S10.55]{BrouwerHvM}\\ 
\rowcolor{black!10} \newtag{13}{tab:S10.55b} & Delsarte dual of Berlekamp-van Lint-Seidel & $(243,110,37,60)$ & $-25$, $2$ & $27/5$, $45$ & $4$, $15$ \cite[\S10.55]{BrouwerHvM} \\ 
\rowcolor{black!10} \newtag{14}{tab:S10.56} & $M_{23}$-graph & $(253, 112, 36, 60)$ & $-26$, $2$ & ${69}/{13}$, ${143}/{3}$ & $4$, $21$ \cite[\S10.56]{BrouwerHvM} \\ 
\rowcolor{black!20} \newtag{15}{tab:S10.57a} & $2^8.S_{10}$-graph & $(256, 45, 16, 6)$ & $-3$, $13$ & $16$, $16$ &  $10$ $16$ \cite[\S10.57]{BrouwerHvM}\\ 
\newtag{16}{tab:S10.57b} & $2^8.(A_8 \times S_3)$-graph & $(256, 45, 16, 6)$ & $-3$, $13$ & $16$, $16$ & $16$, $16$ \cite[\S10.57]{BrouwerHvM} \\ 
\rowcolor{black!10} \newtag{17}{tab:S10.58} & $2^8.L_2(17)$-graph & $(256,102,38,42)$ & $-10$, $6$ & $56/5$, $160/7$ & $8$, $18$ \cite[\S10.58]{BrouwerHvM} \\ 
\rowcolor{black!20} \newtag{18}{tab:S10.61} & McLaughlin & $(275, 112, 30, 56)$ & $-28$, $2$ & $5$, $55$ & $5$, $22$ \cite[\S10.61]{BrouwerHvM}\\ 
\newtag{19}{tab:ExtraspecialClasses4} & Action of $S_4$ on $\rm{PG}(1,17)$  & $(289, 96, 35, 30)$ & $-6$, $11$ & $17$, $17$ & $17$, $17$ OA Graph\\ 
 \newtag{20}{tab:ExtraspecialClasses5} & Action of $S_4$ on $\rm{PG}(1,19)$  & $(361, 144, 59, 56)$ & $-8$, $11$ & $19$, $19$ & $19$, $19$ OA Graph\\ 
\rowcolor{black!20} \newtag{21}{tab:S10.68} &  $G_2(4)$-graph & $(416, 100, 36, 20)$ & $-4$, $20$ & $26$, $16$ & $5$, $16$ \cite[\S10.68]{BrouwerHvM} \\ 
 \newtag{22}{tab:ExtraspecialClasses6} & Sporadic Peisert  & $(529, 264, 131, 132)$ & $-12$, $11$ & $23$, $23$ &  $23$, $23$ \cite[\S10.70]{BrouwerHvM}\\
\newtag{23}{tab:S10.73A} & Liebeck 
& $(625, 144, 43, 30)$ & $-6$, $19$ & $25$, $25$ & $25$, $25$ \cite[\S10.73A]{BrouwerHvM} \\ 
\newtag{24}{tab:S10.73B} & Liebeck
& $(625, 240, 95, 90)$ & $-10$, $15$ & $25$, $25$ & $25$, $25$ \cite[\S10.73B]{BrouwerHvM}\\ 
 \newtag{25}{tab:ExtraspecialClasses11} & Action of $S_4$ on $\rm{PG}(1,27)$  & $(729, 104, 31, 12)$ & $-4$, $23$ & $27$, $27$ & $27$, $27$ OA Graph\\ 
\newtag{26}{tab:ExtraspecialClasses7}  & Action of $S_4$ on $\rm{PG}(1,29)$  & $(841, 168, 47, 30)$ & $-6$, $23$ & $29$, $29$ & $29$, $29$ OA Graph\\ 
\newtag{27}{tab:S4.961} & Action of $S_4$ on $\rm{PG}(1,31)$  & $(961, 240, 71, 56)$ & $-8$, $23$ & $31$, $31$ & $31$, $31$ OA Graph \\ 
\newtag{28}{tab:ExceptionalClasses3}  & Action of $A_5$ on $\rm{PG}(1,31)$ & $(961, 360,139,132)$ & $-12$, $19$ & $31$, $31$ & $31$, $31$ OA Graph\\
\rowcolor{black!20} \newtag{29}{tab:S10.80} & Dodecad-graph & $(1288, 792, 476, 504)$ & $-36$, $8$ & $23$, $56$ & $23$, $24$ \cite[\S10.80]{BrouwerHvM}\\ 
\rowcolor{black!20} \newtag{30}{tab:S10.81} & Conway & $(1408, 567, 246, 216)$ & $-9$, $39$ & $64$, $22$ & $32$, $11$ \cite[\S10.81]{BrouwerHvM}\\ 
\newtag{31}{tab:ExceptionalClasses4} & Action of $A_5$ on $\rm{PG}(1,41)$ & $(1681, 480, 149, 132)$ & $-12$, $29$ & $41$, $41$  &  $41$, $41$ OA Graph\\
\rowcolor{black!20} \newtag{32}{tab:S10.83} & Suzuki & $(1782, 416, 100, 96)$ & $-16$, $20$ & $27$, $66$ & $6$, $66$ \cite[\S10.83]{BrouwerHvM} \\ 
\rowcolor{black!10} \newtag{33}{tab:S10.84} & $2^{11}.M_{24}$-graph & $(2048, 276, 44, 36)$ & $-12$, $20$ & $24$, $256/3$ & $24$, ? \cite[\S10.84]{BrouwerHvM} \\ 
\rowcolor{black!10} \newtag{34}{tab:S10.85} & $2^{11}.M_{24}$-graph & $(2048, 759, 310, 264)$ & $-9$, $55$ & $256/3$, $24$ & $32$, $24$ \cite[\S10.85]{BrouwerHvM} \\ 
 \newtag{35}{tab:S4.2209} & Action of $S_4$ on $\rm{PG}(1,47)$  & $(2209, 1104, 551, 552)$ & $-24$, $23$ & $47$, $47$ & $47$, $47$ OA Graph \\ 
\rowcolor{black!20} \newtag{36}{tab:S10.88} & Conway   & $(2300, 891, 378, 324)$ & $-9$, $63$ & $100$, $23$ & $44$, $12$ \cite[\S10.88]{BrouwerHvM} \\
\rowcolor{black!20} \newtag{37}{tab:S10.89A} & $7^4:(6.O_5(3))$-graph 
& $(2401, 240, 59, 20)$ & $-5$, $44$ & $49$, $49$ & $9$, $49$ \cite[\S10.89A]{BrouwerHvM} \\ 
 \newtag{38}{tab:S10.89B} & $7^4:(6.(2^4:S_5))$-graph
 & $(2401, 480, 119, 90)$ & $-10$, $39$ & $49$, $49$ & $49$, $49$ \cite[\S10.89B]{BrouwerHvM}\\ 
\rowcolor{black!20} \newtag{39}{tab:S10.89C} & $7^4:(3 \times 2.S_7)$-graph
& $(2401,720,229,210)$ & $-15$, $34$ & $49$, $49$ & $17$, $49$ \cite[\S10.89C]{BrouwerHvM} \\
\newtag{40}{tab:S10.89D} & Action of $A_5$ on ${\rm PG}(1, 49)$
& $(2401,960,389,380)$ & $-20$, $29$ & $49$, $49$ & $49$, $49$ \cite[\S10.89D]{BrouwerHvM} \\ 
\rowcolor{black!20} \newtag{41}{tab:S10.90} & $\fis_{22}$-graph & $(3510, 693, 180, 126)$ & $-9$, $63$ & $78$, $45$ & $22$, ? \cite[\S10.90]{BrouwerHvM}\\
\rowcolor{black!20} \newtag{42}{tab:S10.91} & Rudvalis & $(4060, 1755, 730, 780)$ & $-65$, $15$ & $28$, $145$ & $28$, $28$ \cite[\S10.91]{BrouwerHvM}\\ 
\rowcolor{black!20} \newtag{43}{tab:S10.92} & $2^{12}.HJ.S_3$-graph & $(4096, 1575, 614, 600)$ & $-25$, $39$ & $64$, $64$ & $64$, $40$ \cite[\S10.92]{BrouwerHvM} \\ 
\newtag{44}{tab:ExceptionalClasses5} & Action of $A_5$ on $\rm{PG}(1,71)$ & $(5041, 840, 179, 132)$& $-12$, $59$& $71$, $71$ & $71$, $71$ OA Graph\\
\newtag{45}{tab:ExceptionalClasses6} & Action of $A_5$ on $\rm{PG}(1,79)$ & $(6241, 1560, 419, 380)$ & $-20$, $59$ & $79$, $79$ & $79$, $79$ OA Graph\\
\newtag{46}{tab:S10.93} & $3^8.2^{1+6}.O_6^-(2).2$-graph & $(6561, 1440, 351, 306)$ & $-18$, $63$ & $81$, $81$ & $81$, $81$ \cite[\S10.93]{BrouwerHvM}\\ 
\newtag{47}{tab:ExceptionalClasses7} & Action of $A_5$ on $\rm{PG}(1,89)$& $(7921, 2640, 899, 870)$ & $-30$, $59$ & $89$, $89$ & $89$, $89$ OA Graph\\
\rowcolor{black!20} \newtag{48}{tab:S10.94} & $\fis_{22}$-graph & $(14080, 3159, 918, 648)$ & $-9$, $279$ & $352$, $40$ & $64$, ? \cite[\S10.94]{BrouwerHvM}\\ 
\rowcolor{black!10} \newtag{49}{tab:S10.95} & $5^6.4.\mathit{HJ}.2$-graph & $(15625, 7560, 3655, 3660)$ & $-65$, $60$ & $1525/13$, $8125/61$ & ?, ? \cite[\S10.95]{BrouwerHvM} \\ 
\rowcolor{black!20} \newtag{50}{tab:S10.96} & $\fis_{23}$-graph & $(31671, 3510, 693, 351)$ & $-9$, $351$ & $391$, $81$ & $23$, ? \cite[\S10.96]{BrouwerHvM} \\
\rowcolor{black!20} \newtag{51}{tab:S10.97} & $\fis_{23}$-graph & $(137632,28431, 6030, 5832)$ & $-81$, $279$ & $352$, $391$ & $136$, ? \cite[\S10.97]{BrouwerHvM} \\ 
\rowcolor{black!20} \newtag{52}{tab:S10.99} & $\fis_{24}$-graph & $(306936, 31671, 3510, 3240)$ & $-81$, $351$ & $392$, $783$ & $24$, ? \cite[\S10.99]{BrouwerHvM}\\ 
\rowcolor{black!20} \newtag{53}{tab:S10.100} & Suz-graph & $(531441,65520,8559,8010)$ & $-90$, $639$ & $729$, $729$ & $81$, ? \cite[\S10.100]{BrouwerHvM}\\
 \bottomrule
\end{tabular}
\caption{Graphs not belonging to a family. Light grey indicates a fractional Delsarte or Hoffman bound. Dark grey indicates either the Delsarte or Hoffman bound is not attained.}
\label{table:NonFamilyGraphs}
\end{table}
}

\section{Results}

The graphs below are ordered according to Theorem \ref{thm:AlmostSimpleAndAffineGraphs}, where references to definitions of graphs in \cite{BrouwerHvM} are also given.

\subsection{Triangular graphs, polar graphs, and related}

These graphs have almost simple automorphism groups with socle $A_n$ or a classical group.

\begin{theorem}\leavevmode
\begin{enumerate}[\rm (a)]
    \item The triangular graph $T_n$ ($n\ge 4$) is separating if and only if $n$ is odd.\label{lemma:TriangularGraphs}
    \item A polar graph is non-separating if and only if its corresponding polar space has
an ovoid.\label{polar}
    \item Let $\Gamma$ be 
a connected component of the distance-2 graph of the dual polar graph arising from a polar space of rank $5$ and order $(q,1)$. Then $\Gamma$ is separating. \label{lem:ConnectedComponent} 
\end{enumerate}
\end{theorem}

\begin{proof}\leavevmode
\begin{enumerate}[(a)]
\item 
As was observed in \cite[Example 5.8]{Araujo:2017aa},
$T_n$ is the line graph of $K_n$ and so has clique number $n-1$, and its complement has clique number the integer part of $n/2$.
Therefore, $T_n$ is separating if and only if $n$ is odd.

\item See \cite[Theorem 6.8(a)]{Araujo:2017aa}.

\item By \cite[Theorem 2.2.20]{BrouwerHvM}, $\Gamma$ has $\nu = (q^4+1)(q^3+1)(q^2+1)(q+1)$, $k = q(q^2+1)(\frac{q^5-1}{q-1})$, and $s = -q^2-1$. Thus $\frac{\nu s}{s-k}= \frac{q^8-1}{q^3-1}$, which is not an integer for $q\ge 2$,
and so $\Gamma$ is separating by Lemma~\ref{Delsarte}.\qedhere
\end{enumerate}
\end{proof}

We summarise the classical polar spaces with rank at least $2$, or their duals for rank exactly $2$, known to have ovoids in Table \ref{table:Ovoids} and those known to not have ovoids in Table \ref{table:NoOvoids}.
See \cite{PartialOvoidsAndPartialSpreadsOfFiniteClassicalPolarSpaces} for a survey of results on this topic, and \cite{BambergDeBeuleIhringer} for an additional subsequent result.
Note that some polar spaces (in rank 2) occur in dual pairs. They are (i) $W(3,q)$ and $Q(4,q)$; (ii) $H(3,q^2)$ and $Q^-(5,q)$. (Note that the polar spaces $Q^+(3,q)$ and their duals $Q^+(3,q)^D$ are \emph{thin} polar spaces. The automorphism group of $Q^+(3,q)$ is of grid type, and it acts imprimitively on the points of $Q^+(3,q)^D$.)

\begin{table}[H]
\centering
\begin{tabular}{ lll } 
 \toprule
polar space &  conditions \\
 \midrule
$Q(4,q)$, $H(3, q^2)$, $Q^+(3, q)$, $Q^+(3, q)^D$, $Q^+(5, q)$ & None \\
$W(3, q)$ & $q$ even \\
$Q(6, q)$ & $q=3^h$ \\
$Q^+(7, q)$ & $q=3^h$ \\
    & $q=2^h$ \\
    & $q = p^h$, $p \equiv 2$ mod 3, $p$ prime, $h$ odd \\
    & $q$ prime \\
\bottomrule
\end{tabular}
\caption{Polar spaces known to have ovoids.}
\label{table:Ovoids}
\end{table}

\begin{table}[H]
\centering
\begin{tabular}{ lll } 
 \toprule
polar space &  conditions \\
 \midrule
$H(4, 4)^D$, $H(5, 4)$ & None \\
$W(3, q)$ & $q$ odd \\
$Q(6, q)$ & $q$ even \\
    & $q>3$, $q$ prime \\
$W(2n+1, q)$, $Q^-(2n+1, q)$, $H(2n, q^2)$ & $n \geq 2$ \\
$Q(2n, q)$ & $n \geq 4$ \\
$H(2n+1, q^2)$ & $n>q^3-q^2+1$ \\
 & $q=p^h$, $p$ prime, $p^{2n+1} > \binom{2n+p}{2n+1}^2 - \binom{2n+p-2}{2n+1}^2$ \\
$Q^+(2n+1, q)$ & $q=p^h$, $p$ prime, $p^n > \binom{2n+p}{2n+1} - \binom{2n+p-2}{2n+1}$ \\
\bottomrule
\end{tabular}
\caption{Polar spaces known to {\bf not} have ovoids.}
\label{table:NoOvoids}
\end{table}

\medskip

\subsection{Non-singular points graphs}

Here we consider the graphs arising by taking non-singular projective points with respect
to a quadratic or Hermitian form. The rank 3 primitive groups associated with these graphs are of almost simple type.

\begin{theorem}\leavevmode
\begin{enumerate}[\rm (a)]
\item Let $m > 3$. Then $\nug_{m}(2)$ is separating.\label{NU}
\item Let $m\ge 3$. Then $\no^{\epsilon}_{2m}(2)$ is separating for $\epsilon = \pm 1$.\label{NO2}
\item Let $m\ge 3$. Then $\no^\epsilon_{2m}(3)$ is separating for $\epsilon = \pm 1$.\label{NO3}
\item If $m\ge 1$ and $q \in \{4, 8\}$, then $\no_{2m+1}^+(q)$ is separating.\label{NOplusEven}
\item $\no^-_{2m+1}(q)$ is separating for all $m\ge 1$, and prime powers $q>2$.\label{lem:NO-2m+1q}
\item $\no^+_{2m+1}(3)$ is non-separating if and only if $m=2$.\label{lem:NO+5_3}
\end{enumerate}
\end{theorem}

\begin{proof}\leavevmode
\begin{enumerate}[(a)]
    \item The graphs $\overline{\nug_{m}(2)}$ form a tower \cite[\S 3.1.6]{BrouwerHvM}, so their clique number is $m$. The Delsarte bound of $\nug_m(2)$ is $2^{m-1}$. Clearly $m \cdot 2^{m-1} < \frac{2^{m-1}(2^m - (-1)^m)}{3}$ for $m>3$, and so the claim follows from Lemma \ref{Delsarte}.

\item
By \cite[\S3.1.2]{BrouwerHvM}, the valency of {$\no_{2m}^\epsilon(2)$} is $k=2^{2m-2}-1$
and the smallest eigenvalue is $s=- 2^{m - 1} - 1$ (when $\epsilon=1$)
or $s=-2^{m - 2} - 1$ (when $\epsilon=-1$).
If $\epsilon=-1$, then
\[
\frac{k}{s}=-\frac{2^{2(m-1)}-1}{2^{m - 2} + 1}
\]
which is only an integer for $m=3$. 
Moreover, $\no^-_6(2)$ is separating since $\nu=36$,  $\omega=4$, and $\alpha=5$ (cf. \cite[\S10.15]{BrouwerHvM}).
Therefore, by Lemma \ref{Delsarte}, 
$\no^-_{2m}(2)$ is separating for all $m\ge 3$.
We cannot use this argument for $\epsilon=1$ because, in this case,
$\frac{k}{s}=-\frac{2^{2(m-1)}-1}{2^{m - 1} + 1}$
is an integer, and we obtain an upper bound of $2^{m-1}$ on the size of a clique. 
Indeed, by \cite[\S3.1.6]{BrouwerHvM}, the maximum cliques of $\no^+_{2m}(2)$ 
attain this size. Now by \cite[Proposition 3.6.2]{BrouwerHvM},
the coclique number of $\no^+_{2m}(2)$ is
\[
\begin{cases}
m/2,&\text{if }m\equiv 0\pmod{4};\\
m/2-1,&\text{if }m\equiv 1,2\pmod{4};\\
m/2+1,&\text{if }m\equiv 3\pmod{4}.\\
\end{cases}
\]
for all $m\ge 3$. Therefore,
\[
\omega(\no^+_{2m}(2))\alpha(\no^+_{2m}(2))\le 2^{m-1}\cdot (m/2+1),
\]
which is in turn, smaller than the number of vertices of $\no^+_{2m}(2)$,
which is $2^{2m-1}-2^{m-1}$.

\item 
We can apply a similar argument to $\no^\epsilon_{2m}(3)$.
By \cite[\S3.1.3]{BrouwerHvM}, the valency is $k=\tfrac123^{m-1}(3^{m-1}-\epsilon)$
and the smallest eigenvalue is $s=- 3^{m-2}$ (when $\epsilon=1$)
or $s=-3^{m - 1}$ (when $\epsilon=-1$). The order of $\no^\epsilon_{2m}(3)$
is $\nu=\tfrac123^{m-1}(3^m-\epsilon)$.

If $\epsilon=1$, then
\[
\frac{k}{s}=\frac{\tfrac123^{m-1}(3^{m-1}-1)}{- 3^{m-2}}=-\tfrac12(3^{m}-3)
\]
which is an integer. So a clique meeting the Delsarte bound has size $1+\tfrac12(3^{m}-3)=\tfrac12(3^m-1)$.
By \cite[\S3.1.4: Tower and clique sizes]{BrouwerHvM},
the clique number of $\no^\epsilon_{2m}(3)$ is at most $2m$; this does not meet the Delsarte bound, as $m\ge 3$. 

If $\epsilon=-1$, then
\[
\frac{k}{s}=\frac{\tfrac123^{m-1}(3^{m-1}+1)}{-3^{m - 1}}= -\tfrac12(3^{m-1}+1)
\]
which is an integer. So a clique meeting the Delsarte bound has size $1+\tfrac12(3^{m-1}+1)=\tfrac12(3^{m-1}+3)$,
and a coclique meeting the Hoffman bound has size 
\[
\frac{\tfrac123^{m-1}(3^m+1)}{\tfrac12(3^{m-1}+3)}=
\frac{3^{m-2}(3^m+1)}{3^{m-2}+1}.
\]
For $m=3$, we see that $\no^-_6(3)$ has 126 vertices and valency 45.
Moreover, it has clique number 6 and coclique number 15 (by computer).
Since $6\times 15< 126$, it follows that $\no^-_6(3)$ is separating.

So suppose $m\ge 4$. Now $3^{m-2}$ and $3^{m-2}+1$ are coprime,
so if $\frac{3^{m-2}(3^m+1)}{3^{m-2}+1}$ is an integer, then
$3^{m-2}+1$ divides $3^m+1$; this does not meet the Delsarte bound. 

Finally, $\no^+_6(3)$ is separating, by computer.

\item
First, the clique number $\omega(\no_{2m+1}^+(q))$ is $q^m$ \cite[p.87]{BrouwerHvM}.  
The order $\nu$ of $\no_{2m+1}^+(q)$ is $\tfrac{1}{2}q^m(q^m+1)$.  However,
\[
\frac{\nu}{\omega(\no_{2m+1}^+(q))}=\frac{\tfrac{1}{2}q^m(q^m+1)}{q^m}=\tfrac{1}{2}(q^m+1)
\]
which is not an integer when $q$ is even. Therefore,
$\no_{2m+1}^+(q)$ is separating by Lemma \ref{Delsarte}.

\item 
By \cite[\S3.1.4]{BrouwerHvM}, the valency of $\no_{2m+1}^-(2)$ is $k=(q^{m - 1} - 1) (q^m + 1)$
and the smallest eigenvalue is $s=-(q - 2) q^{m - 1} - 1$.
So
\[
\frac{k}{s}=q^{m-1}-1+2q^{m-1}\frac{q^{m-1}-1}{q^m+1-2q^{m-1}}.
\]
Now $q^m+1-2q^{m-1}$ is coprime to $2q^{m-1}$ and hence
$k/s$ is an integer if and only if
$q^m+1-2q^{m-1}$ divides $q^{m-1}-1$; which is impossible
since the former is larger than the latter. Therefore,
$\no^-_{2m+1}(q)$ is separating by Lemma \ref{Delsarte}.

\item
By \cite[\S3.1.4]{BrouwerHvM},
$\no_{2m+1}^+(3)$ is strongly regular with $\nu = \frac{1}{2}3^m(3^m+1)$, $k = (3^{m-1}+1)(3^m-1)$, and $s = -3^{m-1}-1$. Hence the Delsarte bound is $3^m$ and the Hoffman bound is $\frac{1}{2}(3^m+1)$. The Delsarte bound is always met \cite[\S3.1.4]{BrouwerHvM}. The size of the maximum coclique is given by $2m+1$ when $(-1)^m =1$, and $2m$ otherwise (cf. \cite[\S3.1.4 ``Tower and clique sizes'']{BrouwerHvM}). Hence the Hoffman bound is met precisely when $m=2$.\qedhere
\end{enumerate}
\end{proof}

\subsection{Grassmann graphs and $E_{6,1}$ graphs}

The following gives separating rank 3 graphs for the almost simple primitive groups having
socle isomorphic to $\PSL(n,2)$ or $E_6(q)$.

\begin{theorem}\leavevmode
\begin{enumerate}[\rm (a)]
\item Let $n \ge 4$. Then the Grassmann graph $J_q(n,2)$ is separating if and only if $n$ is odd.\label{lem:Grassmann}
\item The $E_{6,1}(q)$ graph is separating for all prime powers $q$.\label{lem:E61}
\end{enumerate}
\end{theorem}

\begin{proof}\leavevmode
\begin{enumerate}[(a)]
\item By \cite[\S3.5.1]{BrouwerHvM},
$J_q(n,2)$ is strongly regular with $\nu = {n \brack 2}$, $k = (q+1)({n-1 \brack 1}-1)$, and $s = -q-1$. Hence the Delsarte bound is ${n-1 \brack 1}$ and the Hoffman bound is $\frac{q^n-1}{q^2-1}$. By \cite[\S3.5.1]{BrouwerHvM}, 
$\omega = {n-1 \brack 1}$ meeting the Delsarte bound. For $n$
even, $\alpha = \frac{q^n-1}{q^2-1}$ and meets the Hoffman bound, so it is non-separating by Lemma \ref{Delsarte}.
By \cite[\S3.5.1]{BrouwerHvM}, $\alpha = \frac{q^n-q^3}{q^2-1} +1$ if $n$ is odd, not meeting the Hoffman bound,
so it is separating by Lemma \ref{Delsarte}.
\item 
By \cite[Proposition 4.9.1]{BrouwerHvM} the $E_{6,1}(q)$ is strongly regular with
$k = (q^3+1)(q^8+q^7+q^6+q^5+q^4+q^3+q^2+q)$ and
$s = -q^3 -1$.
Hence $\omega \le 1 - \frac{k}{s} = q^8 + q^7 + q^6 + q^5 + q^4 + q^3 + q^2 + q + 1$. However, according to \cite[\S 4.9.2]{BrouwerHvM} the clique number is $q^5 + q^4 + q^3 + q^2 +q +1$, and so it is separating by Lemma \ref{Delsarte}.\qedhere
\end{enumerate}
\end{proof}

\subsection{One-dimensional affine graphs}
The corresponding primitive rank 3 groups for the graphs in this section all have affine type. By a recent result of Muzychuk \cite[Theorem 3]{Muzychuk}, a one-dimensional affine rank 3 graph is, up to complement, a van Lint-Schrijver, Paley, or Peisert graph.

\begin{lemma}
The Paley graph $P_q$ of order $q$ is separating if and only if $q$ is non-square.
\end{lemma}

\begin{proof}
The Paley graph $P_q$ of order $q$ requires $q\equiv 1\pmod{4}$ in order
for it to be undirected. If $q=q_0^2$, then the subfield of order $q_0$ forms
a clique, and since $P_q$ is self-complementary, we see that
$\omega(P_q)\alpha(P_q)=q$ in this case. However, if $q$ is not a square, then
$\omega(P_q),\alpha(P_q)<\sqrt{q}$. (See \cite[p.192]{BrouwerHvM}.)
\end{proof}

\begin{lemma}
The Peisert graph $P^*(p^{2t})$ is non-separating if $t$ is odd.
\end{lemma}

\begin{proof}
When $t$ is odd, both clique and coclique numbers equal $p^{t}$ by  \cite[Theorem 5.1]{KP2004}.
\end{proof}

\begin{remark}
It is an open problem whether the clique number of a Peisert graph for $t$ even is strictly
less than the square root of the order. See
\cite[Conjecture 4.3]{Yip}.
\end{remark}

\begin{lemma}
Let $p$ be a prime and let $e$ be an odd prime, with $p$ primitive modulo $e$.
Let $t$ be a positive integer. Then the van Lint-Schrijver graph 
$\vLS(p,e,t)$ is separating if and only if $t$ is even.
\end{lemma}

\begin{proof}
 First, there are $q=p^{(e-1)t}$ vertices in $\vLS(p,e,t)$.
Let $t$ be even. Then the ratio of the largest and smallest eigenvalues can be derived from \cite[Theorem 7.3.2]{BrouwerHvM} and is equal to
\begin{equation}\label{eq1}
\frac{(q-1)/e}{(-1-(e-1)\sqrt{q})/e}=\frac{1-q}{1-\sqrt{q}+e\sqrt{q}}.
\end{equation}
Now $1-q=(1+\sqrt{q})(1-\sqrt{q}+e\sqrt{q})-(1+\sqrt{q})e\sqrt{q}$
and so $1-q$ is divisible by $1-\sqrt{q}+e\sqrt{q}$ if and only if
$(1+\sqrt{q})e\sqrt{q}$ is divisible by $1-\sqrt{q}+e\sqrt{q}$.

Suppose that $(1+\sqrt{q})e\sqrt{q}$ is divisible by $1-\sqrt{q}+e\sqrt{q}$. Then $1-\sqrt{q}+e\sqrt{q}+eq+\sqrt{q}-1$ is divisible by $1-\sqrt{q}+e\sqrt{q}$,
and hence $1-\sqrt{q}-eq$ is divisible by $1-\sqrt{q}+e\sqrt{q}$.
Thus, $1-\sqrt{q}+e\sqrt{q}$ divides $1-\sqrt{q}-eq+\sqrt{q}(1-\sqrt{q}+e\sqrt{q})=1-q$.
So $(e-1)\sqrt{q}+1$ divides $(e-1)(q-1)-\sqrt{q}((e-1)\sqrt{q}+1) = -(e-1+\sqrt{q})$.
In particular, $(e-1)\sqrt{q}+1\le e-1+\sqrt{q}$, which simplifies to
\[
0\ge (e-1)\sqrt{q}+1 -( e-1+\sqrt{q})=(e-2)(\sqrt{q}-1).
\]
This is impossible because $\sqrt{q}\ge 2$ and $e\ge 3$. Hence expression in (\ref{eq1}) is not an integer and so,
by Lemma \ref{Delsarte}, the van Lint-Schrijver graph is separating when $t$ is even.

Instead, let $t$ be odd. Consider the subfield $K:=\GF(\sqrt{q})$ of $\GF(q)$.
Now if $\omega$ is a primitive element of $\GF(q)$, then $\omega^{\sqrt{q}+1}$ generates $K^*$.
Let $S$ be the set of $e$-th powers in $\GF(q)$, that is, the multiplicative subgroup of $\GF(q)^*$
of order $(q-1)/e$. Since $p$ is primitive modulo $e$, we have
$p^{(e-1)/2}\equiv -1\pmod{e}$,
and hence $\sqrt{q}\equiv -1\pmod{e}$, because $t$ is odd. 
Therefore, $\sqrt{q}+1\equiv 0\pmod{e}$
and hence $\omega^{\sqrt{q}+1}\in S$. So every nonzero element of $K$ lies in $S$
and it follows that $K$ is a clique (because it is closed under subtraction). Moreover, $K$ has size $\sqrt{q}$.
Now, let $\sigma \in \GF(q)^*$ such that $\sigma$ is not an $e$-th power. Suppose that $u$ and $v$ are adjacent vertices, then $u\sigma - v\sigma = (u-v)\sigma$ where $u-v$ is an $e$-th power, so $u\sigma$ and $v\sigma$ are not adjacent. Hence $K\sigma$ is a coclique of $\vLS(p,e,t)$ of size $\sqrt{q}$.
Now $|K||K\sigma| = q$ and so $\vLS(p, e, t)$ is non-separating when $t$ is odd.
\end{proof}

\begin{remark}
In the one-dimensional affine case, many examples give $\qi$-groups (cf. \cite{Dixon:2005aa}).
For instance, the rank 3 automorphism groups of Paley graphs are $\qi$-groups (cf. \cite{Dixon:2005aa}).
The smallest degree of a synchronising one-dimensional affine rank 3 group that
is not $\qi$ is $2^8$, and it arises as a group of automorphisms of a  Cayley graph on $\mathbb{F}_{2^8}$
where the joining set is the set of nonzero cubes. This Cayley graph is a certain kind of
van Lint-Schrijver graph.
\end{remark}

\subsection{Grids, bilinear forms graphs}

The automorphism groups of grid graphs are of grid (or wreath product) type, although they may include rank 3 subgroups which are of affine type.
The $n \times n$ grids are just the Hamming graphs $H(2, n)$.

\begin{lemma}
The $n \times n$ grids are non-separating.
\end{lemma}

\begin{proof}
Since the graph is a grid, $\omega = 2$ and $\alpha = n$.
\end{proof}

The automorphism groups of bilinear forms graphs are of affine type.

\begin{lemma}[\cite{HUANG1987191}]\label{lem:BilinearForms}
The bilinear forms graphs $H_q(2, m)$ are non-separating.
\end{lemma}

\subsection{Affine polar graphs}

As a by-product of the results in this section, we provide examples of synchronising non-$\qi$ affine groups to add to the previously known examples due to Huang et al.~\cite{HuangAlt}.

First, $\vo^+(2e,q)$ always has a clique of size $q^e$ (see \cite[Corollary 4]{Orel}). So $\omega(\vo^+(2e,q))=q^{e}$. A coclique
of size $|\vo^+(2e,q)|/\omega$ gives rise to an ovoid of $Q^+(2e+1,q)$,
as was observed\footnote{Another way to observe the connection to ovoids is to first realise
$\vo^+(2e,q)$ as the subgraph of the collinearity graph of $Q^+(2e+1,q)$ induced
by the set of points $\Gamma_2(P)$ opposite a fixed point $P$. Indeed, if $\perp$ is a polarity
of the projective space $\mathbb{P}$
defining $Q^+(2e+1,q)$, then the quotient map via $P$, maps the singular lines incident
with $P$ to the points of a hyperbolic quadric of $P^\perp/P$. The nonsingular lines $PX$, where $X$ lies
in $\Gamma_2(P)$, are mapped to the points of the projective space $\mathbb{P}/P$. Moreover,
they do not lie in the hyperplane $P^\perp/P$ and so they are affine points of $\mathbb{P}/P$.
So this is $\vo^+(2e,q)$: the vertices are affine points, that are adjacent if they span a line
incident with a point of the quadric at infinity.} by Orel:

\begin{theorem}[{\cite[Theorem 2 and Proposition 5]{Orel}}]\samepage
Let $n\ge 2$ be even. Then $\vo^+(n,q)$ is non-separating if and only if 
$Q^+(n+1,q)$ has an ovoid.
\end{theorem}

It is known that for $n\in \{2,4\}$, the polar space $Q^+(n+1,q)$ has an ovoid, but for $n=6$,
we only know that $Q^+(7,q)$ has an ovoid in each of the following cases:
$q$ is prime, $q$ even, $q$ is a power of $3$, or $q=p^h$ where $h$ is odd and $p\equiv 2\pmod{3}$, where $p$ is prime \cite[Table 1 and Section 4.3]{PartialOvoidsAndPartialSpreadsOfFiniteClassicalPolarSpaces}.
Moreover, if $Q^+(n+1,q)$ has an ovoid, then $Q^+(n-1,q)$ has an ovoid (by projection), and 
no ovoid of $Q^+(n+1,q)$ is known to exist for $n\ge 8$.
By a result of Blokhuis and Moorhouse (cf. \cite{PartialOvoidsAndPartialSpreadsOfFiniteClassicalPolarSpaces}), there are no ovoids of $Q^+(n+1,p^h)$ when
$p^{n/2}>\binom{n+p}{n+1}-\binom{n+p-2}{n+1}$ for $p$ prime. So, for example,
$Q^+(9,q)$ does not have an ovoid when $q$ is a power of 2 or 3.

\begin{lemma}
The $\vo_{2m}^-(q)$ graph is separating, where $m\geq 2$ and $q$ a prime power.
\end{lemma}

\begin{proof}
By \cite[\S 3.3.1]{BrouwerHvM} $\nu = q^{2m}$, $k = (q^m +1)(q^{m-1} -1)$, $s = -(q-1)q^{m-1} -1$, and $t=q^{m-1}-1$. Hence we have
\[
\omega \leqslant  
    \frac{q^{2m-1}}{q^m - q^{m-1}+1} \quad \text{ and }\quad
    \alpha \leqslant  
   q(q^m -q^{m-1} +1).
\]
If either inequality is not sharp, then $\omega \alpha < \nu = q^{2m}$ and so the graph would be separating. Hence we require
\[
\omega = \frac{q^{2m-1}}{q^m - q^{m-1}+1}.
\]
However, this is not an integer, since the numerator is $0 \pmod{q^{m-1}}$, but the denominator is $1 \pmod{q^{m-1}}$, which is not possible.
\end{proof}

\begin{theorem}\samepage\label{synch_not_qi}\leavevmode
\begin{enumerate}[\rm (i)]
\item The alternating forms graph $\mathrm{Alt}(5,p^m)$ is separating.
\item The affine half spin graphs $\vd_{5,5}(q)$ graphs are separating.
\item $\vsz(2^{2e+1})$ is separating, for all $e\ge 0$.
\end{enumerate}
\end{theorem}

\begin{proof}\leavevmode
\begin{enumerate}[(i)]
\item
By \cite[Corollary 3.7]{HuangAlt} we have $\alpha=p^5$, while by \cite[Remark 2.3]{HuangAlt}, we have  $\omega=p^4$. Thus $\omega\alpha<p^{10}=\nu$ and so $\mathrm{Alt}(5,p^m)$ is separating.
\item By \cite[\S 3.3.3]{BrouwerHvM}, $\omega = q^4$ which is less than the Delsarte bound of $q^8$ and so $\vd_{5,5}(q)$ is separating by Lemma~\ref{Delsarte}.
\item Let $q=2^{2e+1}$. The graph $\vsz(q)$ has parameters
$(v,k,\lambda,\mu)= \left( q^4,\; (q-1)(q^2+1),\; q-2,\; q(q-1) \right)$.
Moreover, it has eigenvalues $-q(q-1)-1$ and $q-1$. 
However, 
\[
\frac{(q-1)(q^2+1)}{-q(q-1)-1}=-
\frac{q^4-1}{q^3+1}
\]
which is clearly not an integer for any $q\ge 2$, so by Lemma \ref{Delsarte} the graph is separating. \qedhere
\end{enumerate}
\end{proof}

\begin{theorem}
The only graphs from Table \ref{table:NonFamilyGraphs} that are separating are those in Table \ref{table:NonFamilyGraphsSeparating}.
\end{theorem}

\begin{proof}
The Delsarte and Hoffman bounds are computed for each graph in column 5 of Table \ref{table:NonFamilyGraphs}. Those with fractional values are separating by Lemma \ref{Delsarte} and are highlighted in light grey. In many cases the exact clique and coclique numbers are given in column 6 of Table \ref{table:NonFamilyGraphs}. Those graphs for which either the Delsarte or Hoffman bound is not tight are separating by Lemma \ref{Delsarte}. These are highlighted in darker grey. Note that in many cases we are able to explicitly give a reference for the clique and coclique numbers. In one case, for graph \#\ref{tab:S10.57b} of Table \ref{table:NonFamilyGraphsSeparating}, we determined the coclique number by computer. The remaining graphs are
defined by taking a subset $P$ of the parallel classes of the affine plane $\mathrm{AG}(2,q)$ and defining two points to be adjacent if the line joining them is in one of the parallel classes of $P$. Provided $P$ is neither empty nor contains every parallel class, such a graph has $\alpha = \omega = q$, because the points on a line in one of the chosen parallel classes is a $q$-clique and the points on a line not in one of the chosen parallel classes is a $q$-coclique.
As this applies more generally, to any graph arising from an orthogonal array, we have labelled these graphs as $\mathrm{OA}$-graphs. 
\end{proof}

\subsection{Miscellaneous computational results}

While the clique and coclique numbers for many of the rank 3 graphs appearing in this paper were already known (and given in Brouwer and van Maldeghem \cite{BrouwerHvM}) there were some missing values. We used computational search techniques such as integer linear programming, constraint satisfaction programming and direct search to determine some of these previously unknown values. Table~\ref{tab:computations} shows the values that were previously not in the literature for at least one value. The values of $\alpha$ for graphs 48 and 50 were obtained by applying symmetry assumptions and searching for explicit cocliques that meet the Hoffmann bound.

\begin{table}[!ht]
\begin{tabular}{rlcc}
\toprule
\# & Parameters of $\Gamma$ & $\omega(\Gamma)$ & $\alpha(\Gamma)$ \\
\midrule
12   &  $(243, 22, 1, 2)$ & $3$ & $45$\\
13   & $(243, 110, 37, 60)$ & $4$ & $15$\\
15 & $(256, 45, 16, 6)$ & $10$ & $16$ \\
16 & $(256, 45, 16, 6)$ & $16$ & $16$ \\
48 & $(14080, 3159, 918, 648)$ & $64$ & $40$ \\ 
50 & $(31671, 3510, 693, 351)$ & $23$ & $81$ \\ 
\bottomrule
\end{tabular}
\caption{Newly determined values for $\alpha$ and $\omega$.}
\label{tab:computations}
\end{table}

\subsection*{Acknowledgements} This work forms part of an Australian Research Council Discovery Project
DP200101951.

\bibliographystyle{abbrv}

\begin{thebibliography}{10}

\bibitem{ABC}
M.~Aljohani, J.~Bamberg, and P.~J. Cameron.
\newblock Synchronization and separation in the {J}ohnson schemes.
\newblock {\em Port. Math.}, 74(3):213--232, 2017.

\bibitem{Araujo:2017aa}
J.~Ara\'{u}jo, P.~J. Cameron, and B.~Steinberg.
\newblock Between primitive and 2-transitive: synchronization and its friends.
\newblock {\em EMS Surv. Math. Sci.}, 4(2):101--184, 2017.

\bibitem{BambergDeBeuleIhringer}
J.~Bamberg, J.~De~Beule, and F.~Ihringer.
\newblock New non-existence proofs for ovoids of {H}ermitian polar spaces and
  hyperbolic quadrics.
\newblock {\em Ann. Comb.}, 21(1):25--42, 2017.

\bibitem{Bannai:1971}
E.~Bannai.
\newblock Maximal subgroups of low rank of finite symmetric and alternating
  groups.
\newblock {\em J. Fac. Sci. Univ. Tokyo Sect. IA Math.}, 18:475--486, 1971/72.

\bibitem{BBMN76}
B.~J. Birch, R.~G. Burns, S.~O. Macdonald, and P.~M. Neumann.
\newblock On the orbit-sizes of permutation groups containing elements
  separating finite subsets.
\newblock {\em Bull. Austral. Math. Soc.}, 14(1):7--10, 1976.

\bibitem{HallPaigeConjecture}
J.~N. Bray, Q.~Cai, P.~J. Cameron, P.~Spiga, and H.~Zhang.
\newblock The {H}all-{P}aige conjecture, and synchronization for affine and
  diagonal groups.
\newblock {\em J. Algebra}, 545:27--42, 2020.

\bibitem{BrouwerHvM}
A.~E. Brouwer and H.~Van~Maldeghem.
\newblock {\em Strongly regular graphs}.
\newblock Cambridge University Press, 2022.

\bibitem{Cameron1981}
P.~J. Cameron.
\newblock Finite permutation groups and finite simple groups.
\newblock {\em Bull. London Math. Soc.}, 13(1):1--22, 1981.

\bibitem{Cameron:2008aa}
P.~J. Cameron and P.~A. Kazanidis.
\newblock Cores of symmetric graphs.
\newblock {\em J. Aust. Math. Soc.}, 85(2):145--154, 2008.

\bibitem{PartialOvoidsAndPartialSpreadsOfFiniteClassicalPolarSpaces}
J.~De~Beule, A.~Klein, K.~Metsch, and L.~Storme.
\newblock Partial ovoids and partial spreads of classical finite polar spaces.
\newblock {\em Serdica Math. J.}, 34(4):689--714, 2008.

\bibitem{Delsarte:1973aa}
P.~Delsarte.
\newblock An algebraic approach to the association schemes of coding theory.
\newblock {\em Philips Res. Rep. Suppl.}, (10):vi+97, 1973.

\bibitem{Dixon:2005aa}
J.~D. Dixon.
\newblock Permutation representations and rational irreducibility.
\newblock {\em Bull. Austral. Math. Soc.}, 71(3):493--503, 2005.

\bibitem{GodsilMeagher}
C.~Godsil and K.~Meagher.
\newblock {\em Erd\H{o}s-{K}o-{R}ado theorems: algebraic approaches}, volume
  149 of {\em Cambridge Studies in Advanced Mathematics}.
\newblock Cambridge University Press, Cambridge, 2016.

\bibitem{HoffmansRatioBound}
W.~H. Haemers.
\newblock Hoffman's ratio bound.
\newblock {\em Linear Algebra and its Applications}, 617:215--219, 2021.

\bibitem{HuangAlt}
L.-P. Huang, J.-Q. Huang, and K.~Zhao.
\newblock On endomorphisms of alternating forms graph.
\newblock {\em Discrete Math.}, 338(3):110--121, 2015.

\bibitem{HUANG1987191}
T.~Y. Huang.
\newblock An analogue of the {{E}rd\H{o}s-{K}o-{R}ado} theorem for the
  distance-regular graphs of bilinear forms.
\newblock {\em Discrete Math.}, 64(2-3):191--198, 1987.

\bibitem{KantorLiebler:1982}
W.~M. Kantor and R.~A. Liebler.
\newblock The rank {$3$} permutation representations of the finite classical
  groups.
\newblock {\em Trans. Amer. Math. Soc.}, 271(1):1--71, 1982.

\bibitem{KP2004}
A.~Kisielewicz and W.~Peisert.
\newblock Pseudo-random properties of self-complementary symmetric graphs.
\newblock {\em J. Graph Theory}, 47(4):310--316, 2004.

\bibitem{Liebeck:1987}
M.~W. Liebeck.
\newblock The affine permutation groups of rank three.
\newblock {\em Proc. London Math. Soc. (3)}, 54(3):477--516, 1987.

\bibitem{LiebeckSaxl}
M.~W. Liebeck and J.~Saxl.
\newblock The finite primitive permutation groups of rank three.
\newblock {\em Bull. London Math. Soc.}, 18(2):165--172, 1986.

\bibitem{Muzychuk}
M.~Muzychuk.
\newblock A classification of one-dimensional affine rank three graphs.
\newblock {\em Discrete Math.}, 344(7):Paper No. 112400, 3, 2021.

\bibitem{Neumann1976}
P.~M. Neumann.
\newblock The structure of finitary permutation groups.
\newblock {\em Arch. Math}, 27(1):3--17, dec 1976.

\bibitem{Orel}
M.~{Orel}.
\newblock {On Minkowski space and finite geometry}.
\newblock {\em {J. Comb. Theory, Ser. A}}, 148:145--182, 2017.

\bibitem{Skresanov}
S.~V. Skresanov.
\newblock On 2-closures of rank 3 groups.
\newblock {\em Ars Math. Contemp.}, 21(1):Paper No. 8, 20, 2021.

\bibitem{Yip}
C.~H. Yip.
\newblock On maximal cliques of {Cayley} graphs over fields.
\newblock {\em J. Algebr. Comb.}, 56(2):323--333, 2022.

\end{thebibliography}

\end{document}